\documentclass[11pt]{amsart}
\usepackage{fullpage} 
\usepackage{graphicx,pdfsync,amssymb,mathrsfs,amsfonts,amsbsy,color,esint,mathtools,tikz,mdframed}
\usetikzlibrary{positioning}
\usepackage[shortlabels]{enumitem}
\usepackage{hyperref}
\hypersetup{
	colorlinks=true,
	linkcolor=blue,
	filecolor=blue,
	urlcolor=blue,
	citecolor=blue,
}
\usepackage[utf8]{inputenc}
\usepackage[T1]{fontenc}
\usepackage{verbatim}
\newtheorem{theorem}{Theorem}[section]
\newtheorem{corollary}[theorem]{Corollary}
\newtheorem{lemma}[theorem]{Lemma}
\newtheorem{proposition}[theorem]{Proposition}

\theoremstyle{definition}

\newtheorem{remark}[theorem]{Remark}
\newtheorem{assumption}[theorem]{Assumption}
\setlist[enumerate]{itemsep=3pt}
\setlist[itemize]{itemsep=3pt}
\numberwithin{equation}{section}

\definecolor{green}{rgb}{0,0.8,0} 

\newcommand\R{\mathbb{R}}
\newcommand{\HH}{\mathbb{H}}
\newcommand\e{\epsilon}

\newcommand{\la}{\langle}
\newcommand{\ra}{\rangle}

\newcommand\loge{\frac{1}{ |\log \epsilon|}}

\newcommand{\nrm}[1]{\Vert#1\Vert}

\newcommand{\tld}[1]{\widetilde{#1}}

\newcommand{\nnrm}[1]{{\vert\kern-0.25ex\vert\kern-0.25ex\vert #1 
		\vert\kern-0.25ex\vert\kern-0.25ex\vert}}

\newcommand{\supp}{{\mathrm{supp}}\,}

\newcommand{\rd}{\partial}
\newcommand{\nb}{\nabla}

\newcommand{\gmm}{\gamma}
\newcommand{\Gmm}{\Gamma}

\newcommand{\eps}{\epsilon}

\newcommand{\lmb}{\lambda}

\newcommand{\Omg}{\Omega}

\newcommand{\zt}{\zeta}

\newcommand{\bbH}{\mathbb H}

\newcommand{\bbR}{\mathbb R}

\begin{document}

	\title[Dynamics for a single vortex ring]
	{Global dynamics of a single vortex ring} 

    \date{\today}
    
	\author[D. Guo]{Dengjun Guo}
	\address{Academy of Mathematics and Systems Science, Chinese Academy of Sciences, Beijing, China}
	\email{djguo@amss.ac.cn}
	\author[I.-J. Jeong]{In-Jee Jeong}
	\address{Department of Mathematical Sciences and RIM, Seoul National University, Seoul 08826 \\ and School of Mathematics, Korea Institute for Advanced Study, Seoul 02455, Republic of Korea}
	\email{galois.ij@gmail.com}
	\author[L. Zhao]{Lifeng Zhao}
	\address{School of Mathematical Sciences,
		University of Science and Technology of China, Hefei 230026, Anhui, China}
	\email{zhaolf@ustc.edu.cn}
	
	\begin{abstract}
		We study the global-in-time dynamics of vortex rings for the three-dimensional incompressible Euler equations, under the assumption of axisymmetric flows without swirl. 
	
		For a broad class of initial data sharing only the macroscopic invariants with a thin vortex ring, we prove that the vorticity remains sharply concentrated and propagates along the symmetry axis with leading-order speed given by the Kelvin--Hicks formula, providing the first global-in-time validation of the vortex filament conjecture for a single vortex ring arising from generic initial data.
		
		We further identify a universal filamentation mechanism driven by the competition between rapid core translation and slower local induction. This mechanism gives linear-in-time stretching of the vortex support under very general assumptions on the data, yielding dynamical instability of any thin vortex ring configurations in the $W^{2,\infty}$ norm.   
	\end{abstract}

    \renewcommand{\thefootnote}{\fnsymbol{footnote}}
    
	\maketitle
	 
	\section{Introduction}
	
	The study of vortex dynamics in ideal fluids is a classical subject dating back to Helmholtz \cite{Helmholtz1858} and Kelvin \cite{Kelvin1867}. A fundamental structure in this field is the \textit{vortex ring}, a toroidal region of vorticity propagating along its symmetry axis, which we take to be the $z$-direction. For a ring with radius $r_0$, circulation $\mu$, and core thickness $\eps \ll 1$, the classical Kelvin--Hicks formula (\cite{Kelvin1867b,Hicks1883}) gives the propagation speed 
	\begin{equation}\label{eq:Kelvin-speed}
		V = \frac{\mu}{4\pi r_0} \left( |\log \e| + O(1) \right) \mathbf{e}_z,
	\end{equation} where the $O(1)$
	term depends on the specific distribution of vorticity within the ring's cross-section \cite{Fraenkel1970}. The \textbf{vortex filament conjecture} predicts that such concentrated structures remain localized and follow the trajectory governed by \eqref{eq:Kelvin-speed}. While special steadily translating solutions (traveling waves) have been established starting with \cite{Norbury73,FraenkelBerger1974, Benjamin}, the conjecture for \textit{generic} initial data remained largely open.
	
	Regarding this issue, earlier works by Benedetto--Caglioti--Marchioro \cite{BenedettoCagliotiMarchioro2000} and Jerrard \cite{Jer} proved that vorticity remains concentrated near some point for all $t > 0$, if it is initially so. However, rigorous tracking of the concentration point along the $z$-axis, henceforth denoted by $z(t)$, has been limited to very short time scales $t \lesssim |\log \eps|^{-1}$. 
	
	In this paper, we establish the global-in-time dynamics for vortex rings with generic initial vorticity--that is, data constrained only by macroscopic invariants. Our results reveal the following:
	
	\begin{enumerate}
		\item \textbf{Global-in-time verification of the conjecture:} For any initial data sharing the macroscopic invariants of a thin vortex ring, the bulk of the vorticity remains concentrated and obeys the prediction $z(t) \sim \frac{\mu}{4\pi r_0}|\log \eps| t$ for all $t \in \mathbb{R}$.
		\item \textbf{Universal filamentation:} Critical thickness scale of $O(|\log \eps|^{-1})$ is identified: If the initial vortex support radius exceeds this scale, the velocity disparity between the vortex core and the periphery leads to linear-in-time filamentation, namely the creation of a long and thin tail behind the vortex ring. Notably, this filamentation behavior is shown to occur for a broad class of axisymmetric initial data sharing the macroscopic invariants of a thin vortex ring, without requiring proximity to a specific traveling wave. Furthermore, filamentation implies infinite growth of various quantities, including the maximum of the vorticity Hessian. 
	\end{enumerate}
	Our results are obtained for the axisymmetric Euler equations without swirl, which we now explain.

	\subsection{Axisymmetric Euler equation}

	The vorticity formulation of the three-dimensional incompressible Euler equations in $\bbR^{3}$ reads  \begin{equation}\label{eq euler 1}
		\left\{
		\begin{aligned}
			\partial_t \Omega + U\cdot \nabla \Omega=\Omega \cdot \nabla U, & \\
			U=\nabla \times \Delta^{-1}\Omega, &
		\end{aligned}
		\right.
	\end{equation} where $U(\cdot,t): \R^3  \to \R^3 $ and $\Omg(\cdot, t) = \nb \times U(\cdot,t) : \R^3  \to \R^3 $ represent the velocity and vorticity of the fluid, respectively. In this work, we study \textit{axisymmetric flows without swirl} in $\R^3$, for which the velocity and vorticity take the form  
	\[
	U=u^r(r,z,t) \mathbf{e}_r + u^z(r,z,t) \mathbf{e}_z,\qquad \Omega=w(r,z,t) \mathbf{e}_{\theta}
	\]
	in cylindrical coordinates. Under these assumptions, the vorticity equation \eqref{eq euler 1} reduces to
	\begin{equation}\label{eq axeuler}
		\partial_t w + \nabla \cdot (uw)=0
	\end{equation} on the half-plane $\mathbb{H}=\{x=(r,z)  \, : \, r > 0, \, \,\, z \in \R \}$ 
	where $ \nabla \cdot (uw) = \rd_{r} (u^rw) + \rd_{z}(u^zw)$. 
    
	Let us recall three conserved quantities for solutions to \eqref{eq axeuler}. To begin with, given a function $w$ on $\mathbb{H}$, we define its energy $E(w)$ by
	\begin{equation}\label{eq:F}
		\begin{split}
			E(w) 
			:=  \iint_{\bbH \times \bbH } \frac{\sqrt{rr'}}{2\pi} F\left( \frac{|x-x'|^2}{rr'} \right) w(x)w(x') \, dx\,dx'  , \quad F(s) := \int_0^{\pi} \frac{\cos a}{\left( s+2-2\cos a \right)^{\frac12}}\,da.
		\end{split}
	\end{equation}  
	Next, the total vorticity $M_{0}$ and the second momentum $M_{2}$ are defined by \begin{equation*}
		\begin{split}
			M_0(w) := \int_{\bbH} w(x) \,dx,  \qquad M_{2} (w) := \int_{\bbH} r^{2} w(x) \,dx, 
		\end{split}
	\end{equation*} respectively. Under mild regularity assumptions, solutions of \eqref{eq axeuler} conserve $E$, $M_0$, and $M_2$.
	
	\subsection{Main Results}
	We consider a broad class of initial data  characterized solely by their macroscopic invariants. Specifically, we shall refer to a \textit{vortex ring configuration} as any vorticity distribution satisfying the following constraints on total vorticity, second momentum, and energy:
	\begin{equation*}
		\begin{split}
			0 \le \frac{w(x)}{r} \lesssim  \frac{1}{\e^2}, \,   M_0(w)=\mu +o(1), \, M_2(w)=r_0^2 \mu+o(1), \, \mbox{ and }  \, E(w)=\frac{r_0 \mu^2}{2\pi}|\log \e|+O(1), 
		\end{split}
	\end{equation*} where $\e > 0$ is a small parameter.
	A prototypical example of initial data satisfying these conditions is the concentrated blob \begin{equation*}
		\begin{split}
			w_0(x)= \frac{\mu}{\e^2}f\left( \frac{x-x_0}{\e} \right), \qquad x_0=(r_0,z_0)
		\end{split}
	\end{equation*}
	where $f\in C_{c}^{\infty}(\bbH) $ is non-negative and satisfies $\int_{\bbH} f(y) \,dy=1$. The vortex filament conjecture predicts that, for such initial data, the corresponding solution $w(x,t)$ remains concentrated near a point of the form
	$$x(t)=\left( r_0, z_0+\left( \frac{\mu}{4\pi r_0}|\log \e|+O(1) \right)t \right).$$  
	Our first result rigorously establishes the vortex filament conjecture globally in time for vortex rings arising from generic initial vorticity. 
	\begin{theorem}\label{thm dynamical stable}
		Given $r_{0}, \mu, c_{1}, c_{2}, c_{3}, c_{4} > 0$, there exist $\e_0, C_{0} > 0$ depending on these parameters such that the following holds. For any $0< \e \le \e_0$, assume that $w_{0,\e} \in L^{\infty}(\bbH)$ and satisfies:
		\begin{itemize}
			\item[\textup{(i)}] $0 \le \frac{w_{0,\e}(x)}{r} < {c_{1}}{\e^{-2}}$, 
			\item[\textup{(ii)}] $|M_0(w_{0,\e})-\mu| <  {c_{2}}{|\log \e|^{-1}}$,
			\item[\textup{(iii)}] $|M_2(w_{0,\e})-r_0^2 \mu| <  {c_{3}}{|\log \e|^{-1}}$, 
			\item[\textup{(iv)}] $|E(w_{0,\e})-\frac{r_0 \mu^2}{2\pi}|\log \e|| < c_{4}$.
		\end{itemize} 
		Let $w_{\e}(\cdot,t)$ be the corresponding solution to the axisymmetric Euler equation \eqref{eq axeuler}. Then for any $t \in \bbR$, $w_{\e}(\cdot,t)$ remains concentrated near some point $x_{\e}^*(t)=(r_{\e}^*(t),z_{\e}^*(t)) \in \mathbb{H}$ in the sense that 
        \begin{equation*}
			\begin{split}
				\mbox{\textup{(a)}} \quad |r_{\e}^*(t)-r_0| \le C_0|\log \e|^{-1} \quad \mbox{and} \quad \int_{|x-x_{\e}^*(t)|\ge \e^{1/2}} (1+r^2)w(x,t) \,dx \le C_0|\log \e|^{-1}. 
			\end{split}
		\end{equation*}
		If we further assume that the initial vorticity is not too spread out in the $z$-axis, in the sense that 
		\begin{itemize}
			\item[\textup{(v)}] $A_{0,\e}:=\int_{\bbH} \langle z \rangle r^2w_{0,\e}(x) \,dx < +\infty$ holds,
		\end{itemize} then we have for all $t \in \bbR$ that 
		\begin{equation*}
			\begin{split}
				\mbox{\textup{(b)}} \quad \left|z_{\e}^*(t)-\frac{\mu }{4\pi r_0}|\log \e|t\right| \le C_{0}( 1+A_{0,\e}+|t|).
			\end{split}
		\end{equation*}
	\end{theorem}
	\begin{remark}
		We give a few remarks regarding the statements. 
		\begin{itemize}
			\item (\emph{Global wellposedness}) The assumptions for $w_{0,\eps}$ in Theorem \ref{thm dynamical stable} imply in particular that $w_{0,\eps}, \frac{w_{0,\eps}}{r} \in L^{1} \cap L^{\infty}(\bbR^{3})$. Under these assumptions, global existence and uniqueness of the corresponding solution of \eqref{eq axeuler} (and also of \eqref{eq euler 1}) satisfying $w_{\eps}, \frac{w_{\eps}}{r} \in L^{\infty}_{loc}( \bbR ; L^{1} \cap L^{\infty}(\bbR^{3} ))$ is well-known (\cite{UY1968,Danaxi}). This unique solution is \textit{Lagrangian} in the sense that there is a uniquely defined global in time ($t\in\bbR$) flow map $X(\cdot, t) : \bbH \to \bbH$ satisfying \begin{equation*}
				\begin{split}
					\partial_{t} X = u_{\eps}(X(\cdot,t),t), \quad X(x,t) = x, \quad (\frac{w_{\eps}}{r})(X(\cdot,t),t) = \frac{w_{0,\eps}}{r}, 
				\end{split}
			\end{equation*} where $u_{\eps}(\cdot,t)$ is the velocity of $w_{\eps}(\cdot,t)$ defined by \eqref{eq biot velocity}. One can verify that this unique solution conserves $M_{0}, M_{2}$, and $E$, and that the assumptions (i)--(iv) propagate for all times.
			\item (\emph{Scaling invariance}) If $w(x,t) = w(r,z,t)$ is a solution of \eqref{eq axeuler}, then for any $\gamma, \lambda > 0$, the function $w^{\lambda,\gamma}$ defined by 
			$$
			w^{\lambda,\gamma}(r,z,t):=\gamma w(\lambda r,\lambda z, \gamma t)
			$$
			solves \eqref{eq axeuler} as well. Using this scaling invariance, it suffices to prove the result in the case $\mu=r_0=1$, see Subsection \ref{subsec:prelim} below for details. 
			
			\item (\emph{Time reversal invariance}) Similarly, given a solution $w(r,z,t)$, one can check that $w(r,-z,-t)$ provides another solution of \eqref{eq axeuler}. Since all the assumptions (i)--(v) are invariant under $(r,z)\mapsto (r,-z)$, it suffices to prove the statements only for $t\ge0$. 
			
			\item (\emph{Translation invariance}) Except for the assumption on \( A_{0,\e} \), the other four hypotheses are invariant under translations in \( z \)-direction. Since \eqref{eq axeuler} has translation invariance in $z$, the above theorem implies the following: if there exists some \( z_{0,\epsilon} \in \mathbb{R} \) such that 
			\[ \qquad \quad \int_{\bbH} (1 + |z - z_{0,\epsilon}|^2)^{\frac{1}{2}} r^2 w_{0,\epsilon}(r, z) \, dx \lesssim 1 \quad\mbox{holds, then}\quad \left|z_{\e}^*(t)-z_{0,\e}-\frac{\mu }{4\pi r_0}|\log \e|t\right| \lesssim 1+|t|.
			\]
			\item (\emph{Short time dynamics}) 
			If we further assume that $A(0)-r_0^2\mu=o(1)$, then by \eqref{eq sec2 z*} we have 
			\[
			\qquad \left|z_\e^*(t)-\frac{\mu}{4\pi r_0}|\log \e|t\right|\le c_{t,\e,A(0)}, \quad \mbox{where} \quad c_{t,\e,A(0)}\to 0 \quad \mbox{as}\quad (t,\e,A(0)-r_0^2\mu)\to(0,0,0). 
			\] 
			\item (\emph{Sharpness of the axial tracking}) 
			The error bound of $O(t)$ in statement (b) is unavoidable under the assumptions (i)--(v). As demonstrated in the classical work of Fraenkel \cite{Fraenkel1970}, the $O(1)$ correction term in the Kelvin--Hicks formula \eqref{eq:Kelvin-speed} depends explicitly on the specific vorticity distribution (the profile) within the core, which cannot be determined by (i)--(v). 
		\end{itemize}
	\end{remark}
	 
	Our second result identifies a mechanism for filamentation, more specifically linear-in-time  creation of a long and thin tail along the $z$-axis. To quantify this, for a subset $S \subset \bbH$, we define \begin{equation*}
		\begin{split}
			\mathrm{diam}_{z}\left(S\right) = \sup\, \left\{ |z_{1} - z_{2}| \, : \, (r_1,z_1), (r_2,z_2) \in S \right\}. 
		\end{split}
	\end{equation*}
	
	\begin{theorem}\label{thm filamentation without perturbation}
		Let $w_{0,\eps}$ satisfy the assumptions \textup{(i)}--\textup{(v)} from Theorem \ref{thm dynamical stable}. Given $c_5>0$, there exist constants $\eps_{0},\tilde{C}_{d}>0$ such that the following holds for any $0 < \e \le \e_0$. Assuming further that there exists a subset $U \subset \operatorname{supp}(w_{0,\e})$ satisfying
		\begin{equation}\label{eq2co:off-axis-support}
			\bigl|U \cap \{ r \ge c_{5} \} \bigr| \ge \frac{\tilde{C}_{d}}{|\log \eps|^2} \quad \mbox{and} \quad \int_U w_{0,\e}(x)\,dx \ge \frac{\mu}{2},
		\end{equation}
		then the image of $U$ under the flow map $X(\cdot,t)$  satisfies, for all $t \ge 0$, \begin{equation*}
			\begin{split}
				\mathrm{diam}_{z}( X(U,t) ) \ge \frac{\mu}{8\pi r_{0}}|\log\eps|t - C_{\e}
			\end{split}
		\end{equation*} for some $C_{\e}$ depending on $w_{0, \e}$. In particular, the vorticity support radius grows linearly in time.
	\end{theorem}
	\begin{remark}
		The geometric condition \eqref{eq2co:off-axis-support} simply states that if the vortex is too ``fat,'' parts of it will be too far from the center to get enough induction. Since the core's speed $O(|\log \eps|)$ outpaces the weak induction at distance $O(|\log \eps|^{-1})$, these outer parts inevitably lag behind and peel off.
	\end{remark}

	As a direct consequence of filamentation, we obtain the growth of $\|\partial_{rr} w_\eps\|_{L^\infty}$:  as the vortex stretches in the $z$-axis, it becomes thinner in the radial direction by conservation of its support volume, causing growth of the derivatives of the vorticity. 
	\begin{corollary}\label{co filamentation without perturbation}
		Let $w_{ 0, \e} \in C_{c}^{\infty}(\mathbb{H})$ be supported away from $\{ r = 0 \}$ and satisfy assumptions \textup{(i)}--\textup{(v)}. Given $c_5>0$, there exist constants $\eps_{0}, {C}_{2,d}>0$ such that the following holds for any $0 < \e \le \e_0$. Suppose there exists a smooth and simple closed curve $\Gamma_1$ enclosing a region $D_1$ such that 
		$$
		\inf_{x\in \Gamma_1} |w_{0,\e}(x)| >0, \quad \int_{D_1} w_{0,\e}\,dx \ge \frac{\mu}2, \quad \mbox{and} \quad |\operatorname{supp}(w_{0,\e}) \cap D_1 \cap \{ r \ge c_{5} \}| \ge \frac{C_{2,d}}{|\log \e|^2}.
		$$ 
		Then the solution $w_{\e}(\cdot,t)$ satisfies the following lower bound for all $t \ge T$:
		\begin{equation*}
			\nrm{ \rd_{rr} w_{\e}(\cdot,t) }_{L^{\infty}(\bbH)} \ge \eta \sqrt{t}, \qquad \mbox{for some } \eta, T \mbox{ depending on } w_{0,\e}. 
		\end{equation*}  
	\end{corollary}
	
	Finally, we conclude that vortex rings are generically unstable in $W^{2,\infty}$.
	\begin{theorem} \label{thm dynamical unstable}
		Let $w_{\star, 0, \e} \in C_{c}^{\infty}(\mathbb{H})$ satisfy assumptions \textup{(i)}--\textup{(v)} of Theorem \ref{thm dynamical stable}. Then there exists a smooth perturbation term $w_{\star, 0, p}\in C_{c}^{\infty}(\mathbb{H})$ such that the corresponding solution $w_{\e}(x,t)$ with initial datum $w_{0,\e}=w_{\star, 0, \e}+\delta w_{\star, 0, p}$ satisfies
		\begin{equation*}
			\nrm{ \rd_{rr} w_{\e}(\cdot,t) }_{L^{\infty}(\bbH)} \ge \eta \sqrt{t} \qquad\mbox{for all}\qquad t \ge T \quad \mbox{and} \quad 0<\delta \le \delta_0,
		\end{equation*}
		for some $\eta, \delta_0, T>0$ depending on $w_{\star, 0, \e}$.
	\end{theorem}
	\begin{remark}
		Taking $\delta$ small enough, the perturbation $\delta w_{\star, 0, p} = w_{0,\e}-w_{\star, 0, \e}$ becomes arbitrarily small in any reasonable functional spaces.
	\end{remark}
	
	\smallskip 
	
	\noindent\textbf{Technical basis for instability.} The proof of the above universal instability criteria relies on the analysis of configurations admitting a ``core-distorted'' decomposition, serving as the constructive basis. Moreover, the decomposition developed here is robust and may be applicable to other transport equations exhibiting competition between coherent motion and shearing mechanisms.
	
	\begin{theorem}\label{thm filamentation}
		For $r_{0}, \mu, c_{1}, c_{2}, c_{3}, c_{4} > 0$, there exist $\eps_{0},C_{d}>0$ such that the following holds. If $w_{0,\eps}$ satisfies the assumptions \textup{(i)}--\textup{(v)} from Theorem \ref{thm dynamical stable} for some $0<\eps\le\eps_{0}$ and admits a decomposition $w_{0,\eps} = w_{m,0,\eps} + w_{d,0,\eps}$ satisfying \begin{equation}\label{eq:fila-assume}
			\begin{split}
				\quad w_{m,0,\eps}, w_{d,0,\eps} \ge 0 \quad\mbox{and} \quad 0 < \left\Vert \frac{w_{d,0,\eps}}{r} \right\Vert_{L^{\infty}} \le \frac{|\log \eps|^{2}}{C_{d}}\nrm{ w_{d,0,\eps} }_{L^{1}} ,  
			\end{split}
		\end{equation} then the solution $w_{\eps}(\cdot,t)$ corresponding to $w_{0,\eps}$ exhibits  linear-in-time filamentation in the $z$-direction
		\begin{equation*}
			\begin{split}
				\mathrm{diam}_{z}\left( \supp\left( w_{\eps}(\cdot,t) \right)  \right) \ge \frac{\mu}{8\pi r_{0}}|\log\eps|t - C_{\e}
			\end{split}
		\end{equation*} for some $C_{\e}>0$ depending only on $w_{0,\eps}$. 
	\end{theorem}
	
	\begin{remark} We give a few remarks regarding the condition \eqref{eq:fila-assume} in Theorem \ref{thm filamentation}. The filamentation mechanism is driven by vorticity particles whose distance from the vortex center satisfies
			\[
			d \gg |\log \eps|^{-1},
			\]
			for which the induced velocity cannot match the core translation speed of order $|\log \eps|$. Condition \eqref{eq:fila-assume} ensures the presence of a portion of vorticity located beyond the critical induction scale from the vortex center. Indeed, assuming for simplicity that $w_{0,\eps}$ is concentrated near $r=1$, the last inequality in \eqref{eq:fila-assume} yields
			\begin{equation*}
				\begin{split}
					\left\Vert  w_{d,0,\eps}  \right\Vert_{L^{\infty}} 
					\lesssim  C_d^{-1}|\log \eps|^{2} \nrm{ w_{d,0,\eps} }_{L^{1}}
					\lesssim  C_d^{-1}|\log \eps|^{2} 	\left\Vert  w_{d,0,\eps}  \right\Vert_{L^{\infty}}  
					|\supp( w_{d,0,\eps})| ,
				\end{split}
			\end{equation*}
			which implies that $|\supp(w_{d,0,\eps})| \gtrsim \frac{ C_d}{|\log \eps|^2}$, corresponding to the length scale $\frac{\sqrt{C_d}}{|\log \eps|}$; see Figure \ref{fig:data-fila}.
			
			While Theorem \ref{thm filamentation without perturbation} presents a more general criterion, its proof is established by constructing a suitable decomposition of the form \eqref{eq:fila-assume} and applying the estimates developed for Theorem \ref{thm filamentation}. 
	\end{remark}

	\begin{figure}
		\includegraphics{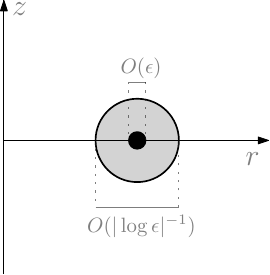}  
		\caption{Illustration for an example of $w_{0,\eps}$ satisfying the assumptions of Theorem \ref{thm filamentation}. The black and gray regions represent the support of $w_{0,m,\eps}$ and $w_{0,d,\eps}$, where the vorticity has size $O(\eps^{-2})$ and $O(1)$, respectively.}\label{fig:data-fila} 	 
	\end{figure}

	\subsection{Ideas of the proof} We now explain the key new ideas in current work. 
    
    \smallskip
    
    \noindent\textbf{Stability beyond finite-time.} The first pillar of our work is the rigorous justification of the Kelvin--Hicks formula \eqref{eq:Kelvin-speed} over infinite time horizons. Previous studies (e.g., \cite{BenedettoCagliotiMarchioro2000, Jer}) typically relied on tracking the movement of the barycenter\begin{equation*}\bar{z}(t) = \frac{ \int_{\bbH} z w(x,t) \, dx}{\int_{\bbH} w(x,t) \,dx}.\end{equation*}
    Its time evolution is governed by the velocity field through the Biot–Savart law. The singularity of the Biot–Savart kernel near the symmetry axis necessitates a radial truncation. However, this operation breaks the symmetry and inevitably generates error terms of the form $\int |z| w \, dx$. While the second moment $M_2 = \int r^2 w \, dx$ is conserved, the axisymmetric Euler equations provide no corresponding control in the axial direction. Consequently, these error terms become uncontrollable for large $|z|$, forcing previous works to impose additional cutoffs in $z$--direction, which restricts the analysis to short time scales $t \lesssim |\log \e|^{-1}$ (see \eqref{eq:Jerrard2} below).

    Beyond these technical obstacles, there is a more fundamental conceptual issue: our filamentation results actually show that after a long time, the barycenter $\bar{z}(t)$ may deviate significantly from the ring core, making it a poor proxy for the ring’s location. 

\smallskip
	
	\noindent\underline{1. The Lagrangian momentum barrier.} 
	To resolve the issues associated with the barycenter, we introduce a weighted ``moving barrier'' \begin{equation*}
		A(t) := \int_{\mathbb{H}} r^2 \langle z - V_{\epsilon}t \rangle w(x, t) \, dx \quad \text{where} \quad \langle z \rangle := \sqrt{1+z^2},
	\end{equation*}
	and develop a stability analysis around this quantity. This functional encodes the persistence of the core structure while remaining robust against the potential formation of filamentation.
    
    The specific choice of the weight $r^2$ is crucial. First, the factor $r^2$ naturally neutralizes the singularity of the Biot-Savart kernel near $r=0$, eliminating the need for any radial truncation. Second, the time derivative $\dot{A}(t)$ exhibits a special structural property: the leading-order term matches precisely with the energy $E$. Moreover, this choice triggers a delicate anti-symmetric structure where all error terms involving $\int |z| w\,dx$ cancel out exactly. Notably, replacing $r^2$ with any other weight would leave behind uncontrollable terms of order $\int |z| w \, dx$. As a result, we were able to prove the uniform-in-time growth bound $\dot{A}(t) \le C$ independently of $\epsilon$, which establishes a fundamental \emph{separation of scales}: while the ring moves at speed $O(|\log \epsilon|)$, any internal deformations or ``leakages'' in the translating frame remain $O(1)$. This allows us to justify \eqref{eq:Kelvin-speed} for all times.

    \smallskip
    
    \noindent\underline{2. Global confinement of the weighted vorticity.} While resolving axial issues, the $r^2$ weight in $A(t)$ introduces new challenges at spatial infinity. Jerrard \cite{Jer} (see also Proposition \ref{prop:Jerrard} below) proved global concentration of the vorticity $w$ around the ring core. However, this estimate does not preclude leakage of vorticity toward large radii. To overcome this, we establish a sharp global-in-time confinement for the weighted vorticity $(1+r^2)w$. By identifying a sharper coupling between the energy $E$, the total vorticity $M_0$, and the second momentum $M_2$, we prove that for all $t \in \bbR$, there exists a point $x^*(t) = (r^*(t), z^*(t))$ satisfying
    \begin{equation*}|r^*(t) - r_0| \lesssim |\log \epsilon|^{-1} \quad \text{and} \quad \int_{|x - x^*(t)| \ge \epsilon^{1/2}} (1+r^2) w(x,t) \, dx \lesssim |\log \epsilon|^{-1}.\end{equation*}
	
	\medskip 
	
	\noindent\textbf{Universal filamentation and geometric threshold}. The second pillar of our work is the discovery of a universal mechanism for filamentation that arises when a significant portion of the vorticity is distributed beyond a critical geometric scale. Unlike classical stability results that often focus on identifying ``unstable perturbations'' for particular steady states, our approach establishes that filamentation is a robust dynamical consequence for a broad class of vortex rings whose initial configurations exhibit a sufficiently large spatial spread.

	\smallskip 
	
	\noindent\underline{1. Balancing induction and translation.} 
	The physical intuition for this instability stems from a competition between global translation and local induction. A vortex ring of core radius $\epsilon$ propagates at a leading-order speed $V_\epsilon \sim |\log \epsilon|$, while the local velocity induced by the core decays as $d^{-1}$ at a distance $d$. In the moving frame of the ring, the structural coherence relies on the local induction being strong enough to synchronize the peripheral fluid with the high-speed core. When the initial vorticity is not sufficiently concentrated within the threshold $d \lesssim |\log \epsilon|^{-1}$, the local induction typically becomes insufficient to bridge this velocity gap. This mismatch suggests that, for generic configurations, the far-field components will fail to keep pace with the core, leading to an inevitable loss of coherence.

	\smallskip 
	
	\noindent\underline{2. Asymptotic stretching and the growth of derivatives.} 
	By utilizing a dual-velocity framework and the measure-preserving property of the Euler flow, we prove that this velocity mismatch leads to a linear-in-time stretching of the vortex support. For initial data where the vorticity is sufficiently spread and bounded away from the symmetry axis, the core effectively outruns its distant support, resulting in the linear growth of the axial diameter $\mathrm{diam}_z(\operatorname{supp} w)$. This geometric stretching is inherently accompanied by an intensification of vorticity gradients, leading to the growth of higher-order norms over time. 

\smallskip

    \noindent We conclude the introduction by reviewing the existing literature on vortex rings and clarifying the relation with previous approaches.
	 
	\subsection{Literature on vortex rings}\label{subsec:literature}
	
	The mathematical analysis of vortex rings in the incompressible Euler equations has developed along several distinct directions, driven by different dynamical questions. Broadly speaking, existing works may be grouped according to whether they address 
	\begin{itemize}
		\item The existence and structure of steady vortex rings,
		\item The interaction dynamics of multiple vortex rings on finite time scales, 
		\item The global-in-time evolution of a single vortex filament under general initial data.
	\end{itemize}
	Each of these questions has motivated specific analytical techniques and has led to results of a fundamentally different nature.
	
	\subsubsection{Construction of steady vortex rings} Classical contributions by Fraenkel and Berger \cite{FraenkelBerger1974} and subsequent developments by Benjamin \cite{Benjamin}, Turkington \cite{Turkington1983}, Burton \cite{Burton} and others employ variational methods to obtain vortex rings as extremizers of suitable energy functionals under {total vorticity and second momentum} constraints; see, for instance, \cite{FriedmanTurkington1981,Burton1988,AF88,BaBu01,Choi2024,CaoQinZhanZou2023}. These variational constructions yield global-in-time solutions, which are uniformly translating vortex rings. Such steady vortex rings can also arise from prescribed-profile constructions, see \cite{Fraenkel1970} in which the leading order of the vorticity is assumed to take the form of a rescaled profile,
	\begin{equation}\label{eq fraenkel profile}
		\omega_\epsilon(x)=\frac{1}{\epsilon^{2}}\zt \!\left(\frac{x-x_0}{\epsilon}\right),
	\end{equation}
	with $\zt$ satisfying specific structural assumptions; see also \cite{Adebiyi1981}. Such approaches provide a detailed description of the internal structure of vortex rings, but apply only to highly constrained classes of initial data and do not address the evolution of general vortex rings. 
	
	A closely related problem in two dimensions is to construct steadily traveling vortices in the half-plane $\bbR^2_+ = \left\{ (x_1,x_2) : x_{2} > 0 \right\}$ (\cite{Burton1996,BurtonNussenzveigHelena2013,Burton2021,AbeChoi2022,GallaySveark2024,Wang2024,CaoQinZhanZou2025,CJS,HuTo,ACJ}). A single point vortex (vorticity given as the Dirac delta distribution) in the half-plane corresponds to the straight-line vortex filament in three dimensions, and more regular traveling wave solutions can be obtained by desingularizing it, similarly to the axisymmetric case.

	\subsubsection{Interaction of multiple vortex rings and point vortices} 
	In the two-dimensional case, Marchioro and Pulvirenti \cite{MarchioroPulvirenti1983,MarchioroPulvirenti1994} developed a flexible framework that controls vorticities which are highly concentrated near a few points, without prescribing specific profiles; see also the subsequent improvements in \cite{Marchioro1998,ButtaMarchioro2018}. By propagating assumptions on the vorticity support, this approach derives that the effective dynamics on the time interval (depending on the concentration scale $\epsilon$) during which the support of the vorticity is strongly confined.
	Extending this framework to vortex ring solutions, the works \cite{ButtaMarchioro2022,ButtaMarchioro2020,CavallaroMarchioro2021,ButtaCavallaroMarchioro2025,Donati2025} control the growth of the vorticity support and validate reduced dynamical descriptions on time scales of order $t \lesssim |\log \epsilon|^{-1}$. There are very recent works that construct leapfrogging solutions \cite{HHM1,HHM2} and \cite{DavidaDelpinoMussoWei2024}, using KAM theory and gluing method, respectively.   
	
	\subsubsection{Global dynamics of a single vortex ring with general data}
	A third perspective is provided by energy-based concentration methods, based on the observation that energy is larger for more concentrated vortices. In two-dimensional domains, Turkington \cite{Turkington1987} showed that energy concentration implies global-in-time concentration of a single vortex for a broad class of initial data. Benedetto--Caglioti--Marchioro \cite{BenedettoCagliotiMarchioro2000} extended the result to the case of a single vortex ring. These results establish that vorticity remains highly concentrated in a global sense, without relying on prescribed profiles or strong assumptions on the vorticity support. However, while concentration \textit{at some point in the domain} is guaranteed for all times, precise tracking of that point is not available, even in the two-dimensional case. For a vortex ring, the quantitative description of the concentration point is typically available only when $t \lesssim |\log\e|^{-1}$, see Proposition \ref{prop:Jerrard} below. In the case of the helical Euler equations, Guo--Zhao \cite{GuoZhao2024} recently extended this time scale to $t \lesssim 1$. 
    
	\subsubsection{Filamentation and infinite time norm growth}
	Numerical simulations reveal that support confinement for concentrated vortices does not hold for all times, by creation of long filaments. This is the main obstruction for controlling the concentration points at infinite time. However, at the same time, it is not straightforward to rigorously prove the occurrence of such filamentation.
	
	In two-dimensional bounded domains, pioneering works by Denisov \cite{Den09}, Kiselev--Sverak \cite{KS}, Xu \cite{Xu}, and Zlatos \cite{Zlatos2015} obtained fast growth of the vorticity gradient using perturbations of a large-scale stable vortex with odd/even symmetries. In the half-plane case, Iftimie--Sideris--Gamblin \cite{ISG99,Iftimie1} obtained linear growth of the vorticity support diameter, using additional odd symmetry in $x_{1}$, see also \cite{ILN2003} for confinement results. With this odd symmetry, the sharp vorticity gradient growth problem was settled recently in Zlato\v{s} \cite{zlatos2025}. Without this odd symmetry, \cite{CJ-Lamb,JYZ} obtained growth of vorticity gradient and support diameter for certain perturbations of the Lamb dipole, which is a stable traveling wave in the half-plane (\cite{AbeChoi2022}). See \cite{ChoiJeong2022,DEJ,DE,Ki-sur2,MR25} for more results on filamentation in two dimensions. 
	
	There have been parallel results for the three dimensional case, mostly for axisymmetric solutions without swirl. Choi--Jeong \cite{ChoiJeong2023} obtained infinite time growth of the $W^{2,\infty}$-norm near Hill's vortex, which is an explicit traveling wave solution to \eqref{eq axeuler}. Furthermore, consideration of solutions to \eqref{eq axeuler} satisfying odd symmetry in $z$ led to several results on growth of vorticity in various norms (\cite{Elgindi21,EGM,CJ-axi,KJ22,CMZ25,GMT,EY25}): this models the case of two opposite-signed vortex rings (termed ``anti-parallel''), and growth occurs when the rings approach each other.

	\subsubsection{Contribution of the present work}
	
	To summarize, while substantial progress has been made on the existence of traveling wave configurations and short-time dynamics of vortex rings, their global-in-time behavior under general initial data remains poorly understood, filamentation being the main obstruction. In particular, it was unclear whether a concentrated vortex ring propagates along the $z-$axis with the expected velocity for all times.
	
	On the other hand, proof of filamentation behavior was mostly done under either an additional odd symmetry assumption on the vorticity or for specific perturbations of stable vortex rings. It was unclear whether filamentation could be obtained for vortex ring type data under assumptions only on macroscopic quantities on the data. 
	
	The present work successfully addresses both of these issues, by developing a framework which allows one to track the location of the vortex ring for all times based solely on a few macroscopic quantities. This framework is robust enough to capture the presence of ``slow part'' of the vortex ring solution, which necessarily filaments behind the vortex core.

	\subsection{Notation} We list the notation and conventions used in the paper. 
	\begin{itemize}
		\item The integrals are taken on $\mathbb{H} := \{x=(r,z)  \, : \, r > 0, \, \,\, z \in \R \}$, unless otherwise specified. 
		\item We use the Japanese bracket notation $\la x \ra =\sqrt{1+|x|^2}$ and write $s = |x-x'|/\sqrt{rr'}$.  \item The notation $a \lesssim b$ means that there exists a constant $C>0$ independent of $x,y,\e$ such that $a \le Cb$. The value of $C$ may change from a line to another, and even within a single line. Additionally, $a=O(b)$ denotes that $|a| \lesssim |b|$. We also use $a=o(b)$ or $a \ll b$ to indicate that $|a|$ is much smaller than $|b|$. 
		\item In what follows, for simplicity we drop the subscript $\eps$ from  $w_{0}$ and $w(\cdot,t)$.  
	\end{itemize}

	\subsection{Organization of the paper} In the rest of the paper, we prove Theorem \ref{thm dynamical stable} in Section \ref{sec:stable} and Theorems \ref{thm filamentation without perturbation}, \ref{thm dynamical unstable}, \ref{thm filamentation} and Corollary \ref{co filamentation without perturbation}  in Section \ref{sec:unstable}. Some open questions are listed in Section \ref{sec:open}. A few technical lemmas are proved in the appendix.

	\subsection*{Acknowledgments}  IJ is supported by the NRF grant from the Korea government (MSIT), No. 2022R1C1C1011051, RS-2024-00406821, and the Asian Young Scientist Fellowship. LZ is supported by National Natural Science Foundation of China under contract No. 12271497 and No. 12341102.
    DG would like to thank Xiaoyutao Luo for helpful comments and references.

	
	\section{Concentration}\label{sec:stable} 
	

	The main purpose of this section is to prove Theorem \ref{thm dynamical stable}.
	

	\subsection{Preliminaries}\label{subsec:prelim} 
	The velocity field $u = (u^{r}, u^{z})$ for axisymmetric Euler equation \eqref{eq axeuler} is defined by integral against $K = (K^{r}, K^{z})$: 
	\begin{equation}\label{eq biot velocity}
		u(x,t)=  \int_{\HH} K(x,x') w(x',t)\,dx', \qquad K(x,x') := \frac{1}{2\pi} F_1(s) \frac{(x-x')^{\perp}}{r\sqrt{rr'}}+\frac{1}{2\pi}F_2(s) \frac{1}{r}\sqrt{\frac{r'}{r}}\begin{pmatrix}  0 \\1 \end{pmatrix},
	\end{equation}
	where $x^{\perp} := (z,-r)$, $s := \sqrt{\frac{|x-x'|^2}{rr'}}$, and $F_{1}, F_{2}$ are given by
	\begin{equation} \label{eq:def-F_1-F_2} 
		F_1(s) := \int_0^{\pi} \frac{\cos a}{\left( s^2+2-2\cos a \right)^{\frac32}}\,da, \qquad 
		F_2(s) := \int_0^{\pi} \frac{1-\cos a}{\left( s^2+2-2\cos a \right)^{\frac32}}\,da. 
	\end{equation}
	We will use a few simple estimates for the kernels, see \cite[Section 2]{FS} for a proof.  
	\begin{lemma}\label{le se1 Biot savart estimates} We have the following estimates for the kernels $F$, $F_1$ and $F_2$ given in \eqref{eq:F} and \eqref{eq:def-F_1-F_2}: \begin{equation*}
			\begin{split}
				F(s) = \left\{
				\begin{aligned}
					\frac12 \log\frac{1}{s}+ \log 8 -2+O(s \log s), \quad  & s\le 1\\
					O(\frac{1}{s}), \quad  & s \ge 1
				\end{aligned}
				\right. 
			\end{split}
		\end{equation*}  and
		\begin{equation*}\label{eq:F1F2}
			F_{1}(s) = \left\{
			\begin{aligned}
				\frac{1}{s^2}-\frac38 \log\frac{1}{s}+ O(1) , \quad  & s\le 1\\
				O(\frac{1}{s^{3}}), \quad  & s \ge 1,
			\end{aligned}
			\right. \qquad  F_{2}(s) =  \left\{
			\begin{aligned}
				\frac12 \log\frac{1}{s}+ \frac{3\log 2-1}{2} +O(s^2), \quad  & s\le 1\\
				O(\frac{1}{s^{3}}), \quad  & s \ge 1.
			\end{aligned}
			\right.
		\end{equation*}  
	\end{lemma}

	\medskip
	
	\noindent \textbf{Reduction by scaling}. Using scaling invariance, we show here that the assumptions (i)--(v) from Theorem \ref{thm dynamical stable} can be reduced to the following: 
	\begin{assumption}\label{as initial data0}
		The initial data $w_{0}$ satisfies, for some $c_{0}>0$ independent of $\eps$, \begin{equation*}
			\begin{split}
				0 \le \frac{w_{0}(x)}{r} \le \frac{c_{0}}{\e^2}, \quad M_{0}(w_{0}) = \int_{\bbH} w_{0}(x)\,dx=1, \quad M_2(w_{0})= \int_{\bbH} r^2 w_{0}(x)\,dx = 1
			\end{split}
		\end{equation*} and \begin{equation*}
			\begin{split}
				\left|E(w_{0}) - \frac{1}{2\pi} |\log\e|\right| \le c_{0}.
			\end{split}
		\end{equation*}

	\end{assumption}
	\begin{assumption}\label{as initial data}
		We assume that
		$$\int_{\bbH} \langle z \rangle r^2w_0(x)\,dx <+\infty.$$ 
	\end{assumption}

	\noindent To see this reduction, assume that $w_{0}$ satisfies (i)--(iv) from Theorem \ref{thm dynamical stable} with some constants $r_{0}, \mu, c_{1}, c_{2}, c_{3}, c_{4}$. We introduce $\tld{w}_{0}$ by \begin{equation*}
		\begin{split}
			\tld{w}_{0}(x) = \gmm w_{0}(\lmb x), \qquad\mbox{where}\qquad \gamma := \frac{ \int_{\bbH} r^{2} w_{0}(x) \, dx  }{ \left( \int_{\bbH} w_{0}(x) \, dx  \right)^{2} }, \qquad \lmb^{2} := \frac{ \int_{\bbH} r^{2} w_{0}(x) \, dx  }{  \int_{\bbH} w_{0}(x) \, dx  }. 
		\end{split}
	\end{equation*}  It can be easily checked that \begin{equation*}
		\begin{split}
			\gamma = \frac{r_{0}^{2}}{\mu} + O(\loge), \qquad \lambda = r_{0} + O(\loge)
		\end{split}
	\end{equation*} and that $\tld{w}_{0}$ satisfies Assumption \ref{as initial data0} with some $c_{0}$ depending on $r_{0}, \mu, c_{1}, c_{2}, c_{3}, c_{4}$. Similarly, if $w_{0}$ satisfies (v) from Theorem \ref{thm dynamical stable}, then $\tld{w}_{0}$ satisfies Assumption \ref{as initial data} trivially.
	
	If we denote by  $\tld{w}(x,t)$ and $w(x,t)$ the solutions corresponding to $\tld{w}_0$ and $w_0$, respectively, then we have $\tld{w}(x,t) = \gamma w(\lmb x, \gamma t)$. Therefore, if we prove (a) and (b) from Theorem \ref{thm dynamical stable} for $\tld{w}_{0}$ and $\tld{w}(x,t)$, it implies the statement for $w_{0}$ and $w(x,t)$. 
	
	
	\subsection{Concentration and confinement in $r$}\label{subsec:con-r} In this subsection, we prove part (a) of Theorem \ref{thm dynamical stable}, under Assumption \ref{as initial data0}. 

	\begin{theorem}\label{thm sec2 main0} For $w_0$ satisfying Assumption \ref{as initial data0}, let $w(x,t)$ be the corresponding solution to \eqref{eq axeuler} with initial data $w_0$. Then there exists $\e_0>0$ such that for any $0<\e \le \e_0$, $ 0\le t$ and $1 \le R \le \frac{1}{\e}$, 
		\begin{equation}\label{eq sec20 concentration99}
			\iint_{|x-x'|\ge R\e}(1+r^2)w(x,t)w(x',t)\,dx\,dx' \lesssim \frac{1}{\log R}.
		\end{equation}
		Moreover, there exists $x^*(t)=(r^*(t),z^*(t)) \in \mathbb{H}$ such that
		\begin{equation}\label{eq sec20 A20}
			|r^*(t)-1|\lesssim \frac{1}{|\log \e|} \qquad \mbox{and} \qquad \int_{|x-x^*(t)|\ge R\e } (1+r^2)w(x,t)\,dx \lesssim \frac{1}{\log R}.
		\end{equation} 	
	\end{theorem}
	This is a strengthening of the following result by Benedetto--Caglioti--Marchioro and Jerrard. This improvement is essential for tracking the $z$-coordinate of the vortex ring. \begin{proposition}[\cite{BenedettoCagliotiMarchioro2000,Jer}]\label{prop:Jerrard}
		For $w_0$ satisfying Assumption \ref{as initial data0},  there exists $x^*(t)=(r^*(t),z^*(t))$ for any $t \ge 0$ such that the corresponding solution $w(x,t)$ to \eqref{eq axeuler} satisfies 
		\begin{equation*}\label{eq:Jerrard1}
			|r^*(t)-1| \lesssim \frac{1}{\sqrt{|\log \e|}} \qquad \mbox{and} \qquad \int_{|x-x^*(t)| \ge \e^{\frac{1}{2}}} w(x,t) \,dx \lesssim \frac{1}{|\log \e|}.
		\end{equation*}
		Moreover, for any $T>0$, there exists $\e_0=\e_0(T)>0$ such that for $0<\e \le \e_0$, 
		\begin{equation}\label{eq:Jerrard2}
			|z^*(t)-V_{\e}t| \ll 1 \qquad\mbox{for any}\qquad 0 \le t \le \frac{T}{|\log \e|}. 
		\end{equation}  
	\end{proposition}

	We need a few preliminary lemmas, which are proved in the Appendix. 
	\begin{lemma}[\cite{Jer}]\label{le app B1}
		For $w_0$ satisfying Assumption \ref{as initial data0}, let $w(t)$ be the corresponding solution to   \eqref{eq axeuler}. Then there exists $C>0$ independent of $\e$ such that
		$$
		E_1(w):=\frac{1}{2\pi} \iint_{ \bbH \times \bbH } \log\frac{1}{|x-x|'}\mathbf{1}_{\{|x-x'|\le 1\}} \sqrt{rr'}w(x)w(x')\,dx\,dx' \ge E(w(t)) - C. 
		$$ 
	\end{lemma}
	
	\begin{lemma}[Feng--Sverak type inequality, cf. \cite{FS}]\label{le feng-sverak}
		There holds
		$$
		\|Gf\|_{L^{\infty}} \lesssim \|r^2f\|_{L^1}^{\frac14}\|f\|_{L^1}^{\frac14}\|\frac{f}{r} \|_{L^{\infty}}^{\frac12} 
		$$
		where the operator $G$ is defined by 
		$$Gf(x) := \int_{\bbH} \log \frac{1}{|x-x'|}\mathbf{1}_{\{|x-x'|\le 1\}}f(x')\,dx'.$$
	\end{lemma}
	
	\begin{proof}[Proof of Theorem \ref{thm sec2 main0}] 	Since all the requirements from Assumption \ref{as initial data0} are conserved in time, it suffices to prove \eqref{eq sec20 concentration99} and \eqref{eq sec20 A20} at $t = 0$. We denote $w_{0}$ by $w$ below.
		Define \begin{equation*}
			\begin{split}
				E_2(w) &:=\frac{1}{2\pi} \iint  \log\frac{1}{|x-x'|}\mathbf{1}_{\{|x-x'|\le 1\}} \frac{2+r^2+(r')^2}{4}w(x)w(x')\,dx\,dx' \\
				& = \frac{1}{2\pi} \iint  \log\frac{1}{|x-x'|}\mathbf{1}_{\{|x-x'|\le 1\}} \frac{1+r^2}{2}w(x)w(x')\,dx\,dx'.
			\end{split}
		\end{equation*} 
		Using $\frac{2+r^2+(r')^2}{4} \ge \sqrt{rr'}$, Lemma \ref{le app B1}, and Assumption \ref{as initial data0}, we obtain
		$$
		E_2(w) \ge E_1(w)\ge E(w) - C \ge \frac{|\log \e|}{2\pi} \iint \frac{1+r^2}{2}w(x)w(x')\,dx\,dx'-C,
		$$
		which yields
		\begin{equation*}
			\iint \left( \log \e +\log\frac{1}{|x-x'|}\mathbf{1}_{\{|x-x'|\le 1\}} \right) (1+r^2) w(x)\,w(x')\, dx \ge -C.
		\end{equation*}
		Therefore, for any $1 \le R \le \frac{1}{\e}$, there holds
		\begin{equation*}
			\begin{aligned}
				&\iint_{|x-x'|\le R\e}\log\frac{\e}{|x-x'|}\mathbf{1}_{\{|x-x'|\le 1\}}(1+r^2)w(x)w(x')\,dx\,dx' + C\\ 
				& \qquad \ge
				|\log\e| \iint_{|x-x'|\ge 1} (1+r^2)w(x)w(x')\,dx\,dx'  + \log R\iint_{1\ge |x-x'|\ge R\e} (1+r^2)w(x)w(x')\,dx\,dx' 
			\end{aligned}
		\end{equation*}
		which gives
		\begin{equation}\begin{aligned} \label{eq:inter}
				\log R \iint_{|x-x'|\ge R\e} (1+r^2)w(x)w(x')\,dx\,dx' 
				\le  \iint\log\frac{\e}{|x-x'|}\mathbf{1}_{\{|x-x'|\le R\e\}}(1+r^2)w(x)w(x')\,dx\,dx'+C.\end{aligned}\end{equation}
		Set $w_{\e}(x)=\e^2 w(\e x)$, note that $\|w_{\e}\|_{L^1} \lesssim1$, $\|r^2w_{\e}\|_{L^1}\lesssim \e^{-2}$, and $\left\|\frac{w_{\e}}{r} \right\|_{L^{\infty}} \lesssim \e$, using the Feng--Sverak type inequality (Lemma \ref{le feng-sverak}), we get for any $x$ a uniform bound 
		\begin{equation*}
			\begin{aligned}
				&\int\log\frac{\e}{|x-x'|}\mathbf{1}_{\{|x-x'|\le R\e\}}w(x')\,dx'=\int\log\frac{1}{|\frac{x}{\e}-x'|}\mathbf{1}_{\{|x-\e x'|\le R\e\}}w_{\e}(x')\,dx'\\
				&\qquad \le \int\log\frac{1}{|\frac{x}{\e}-x'|}\mathbf{1}_{\{|\frac{x}{\e}- x'|\le 1\}}w_{\e}(x')\,dx'
				=G(w_{\e})(\frac{x}{\e}) \lesssim 1.
			\end{aligned}
		\end{equation*}
		Thus, from \eqref{eq:inter} we obtain \eqref{eq sec20 concentration99}: \begin{equation*}
			\begin{split}
				\iint_{|x-x'|\ge R\e} (1+r^2)w(x)w(x')\,dx\,dx' \le \frac{C}{\log R}.
			\end{split}
		\end{equation*} 
		Since $\int w(x)\,dx =1+O(\e)\ge \frac12$, there exists $x^* = (r^*,z^*) \in \mathbb{H}$ such that
		\begin{equation}\label{eq sec20 con near r} 
			\int_{|x-x^*|\ge R\e } (1+r^2)w(x)\,dx \le \frac{2C}{\log R}, 
		\end{equation} which is the second estimate in \eqref{eq sec20 A20}. It remains to prove the first estimate from \eqref{eq sec20 A20}. A direct calculation, combined with \eqref{eq sec20 con near r}, gives
		\begin{equation}\begin{aligned}\label{eq sec20 r decompose}
				\tilde{r}&:=\int rw(x)\,dx  =\int_{|r-r^*|\le \e^{\frac12}} rw(x)\,dx +\int_{|r-r^*|\ge \e^{\frac12}} rw(x)\,dx \\
				&= r^*\int_{|r-r^*|\le \e^{\frac12}} w(x)\,dx+O(\loge).
		\end{aligned}\end{equation}
		Combining this with 
		\begin{equation*}\label{eq sec20 upper bound for r}
			\tilde{r} \le \left(\int w(x)\,dx\right)^{\frac12}\left(\int r^2w(x)\,dx\right)^{\frac12} \le 1+O(\loge),
		\end{equation*}
		we see that $r^* \le 2$. Applying the equality
		\begin{equation}\begin{aligned}\label{eq sec20 r equal}
				\tilde{r}^2+\int |r-\tilde{r}|^2 w(x)\,dx = \int r^2w(x)\,dx=1+O(\loge)
		\end{aligned}\end{equation}
		and noting that $p=\tilde{r}$ is the minimizer of the functional $Q(p):=\int |r-p|^2 w(x,t)\,dx$, we see that
		$$
		\int |r-\tilde{r}|^2 w(x)\,dx \le \int |r-r^*|^2 w(x)\,dx.
		$$
		Therefore, \eqref{eq sec20 con near r} gives (recalling that $r^* \le 2$)
		\begin{align*}
			\int |r-\tilde{r}|^2 w(x)\,dx \le& \int_{|x-x^*|\le \e^{\frac12}} |r-r^*|^2 w(x)\,dx+\int_{|x-x^*|\ge \e^{\frac12}} |r-r^*|^2 w(x)\,dx\\
			\lesssim& \, \int_{|x-x^*|\ge \e^{\frac12}} (r+2)^2 w(x)\,dx + O(\loge) 
			\lesssim \frac{1}{|\log \e|}.
		\end{align*}
		Together with \eqref{eq sec20 con near r}, \eqref{eq sec20 r decompose}, \eqref{eq sec20 r equal} and Assumption \ref{as initial data0}, we conclude $
		r^*=1+O(\loge).$ 
	\end{proof}
	
	\subsection{Control of the vorticity outside $z^*(t)$.}
	
	In this subsection, we will prove that $z^*(t)=V_{\e}t+O(t)+O(\loge)$ with $V_{\e}:=\frac{1}{4\pi} |\log \e|$, under the additional assumption \eqref{as initial data}. Introducing $W$ by the transform $w(x,t)=W(r,z-V_{\e}t,t)$, one can readily check that 
	\begin{equation*}\label{eq 3euler}\left\{\begin{aligned}
			&\partial_t W-V_{\e}\partial_zW+\nabla \cdot (UW)=0 \\
			&W(x,0)=w_0(x)
		\end{aligned}\right. \qquad \mbox{with} \qquad U(x,t) := \int_{\HH} K(x,x')W(x',t)\,dx'
	\end{equation*} 
	and the kernel $K(x,x')$ is given in \eqref{eq biot velocity}. From Theorem \ref{thm sec2 main0}, we immediately obtain  
	\begin{equation}\label{eq sec2 A0}
		\int_{|x-x_*(t)|\ge \e^{\frac12}} (1+r^2)W(x,t) \,dx \le \frac{C_0}{|\log \e|}, \quad \mbox{for} \quad  x_*(t):=(r_*(t),z_*(t))=(r^*(t),z^*(t)-V_{\e}t)
	\end{equation} for any $t \ge 0$. 
	We begin with a technical lemma. 
	
	\begin{lemma}\label{le app2 A0}
		Assuming $W(x,t)$ is a non-negative function satisfying  $$\frac{W(x,t)}{r} \lesssim \frac{1}{\e^2} \qquad \mbox{and}\qquad \int_{\bbH} (1+r^2)W(x,t)\,dx \lesssim 1,$$  then for any non-negative function $g(x,t)$, there holds
		\begin{equation} \label{eq:le app2 A0}
			\left|\iint_{\bbH \times \bbH } F_2(s) \sqrt{rr'} g(x,t)W(x',t) \,dx \,dx'\right| \lesssim |\log\epsilon | \int_{\bbH} (1+r^2)g(x,t)\,dx.
		\end{equation}
	\end{lemma}
	
	\begin{proof}
		Denoting the left hand side of \eqref{eq:le app2 A0} by $I$, from Lemma \ref{le se1 Biot savart estimates}, we estimate  
		\begin{equation*}\begin{aligned}
				I & \lesssim \iint_{s \le1} (\frac12 \log \frac{1}{s}+1) \sqrt{rr'}g(x,t)W(x',t)\,dx\,dx' 
				+\iint_{s \ge1} \sqrt{rr'}g(x,t)W(x',t)\,dx\,dx'\\
				& =\iint_{s \le1} (\frac12 \log \frac{1}{|x-x'|}+\log \sqrt{rr'}) \sqrt{rr'}g(x,t)W(x',t)\,dx\,dx' 
				+\iint  \sqrt{rr'}g(x,t)W(x',t)\,dx\,dx' \\
				&\lesssim \iint_{|x-x'| \le1} \log \frac{1}{|x-x'|} \sqrt{rr'}g(x,t)W(x',t)\,dx\,dx' + \iint  (1+\sqrt{rr'})\sqrt{rr'}g(x,t)W(x',t)\,dx\,dx' \\
				& =: I_1+I_2.
		\end{aligned}\end{equation*}
		A direct calculation shows
		$$
		I_2 \lesssim \int (1+r^2)g(x,t)\,dx \int (1+r^2)W(x,t)\,dx \lesssim \int (1+r^2)g(x,t)\,dx.
		$$
		For $I_1$, we split\begin{equation*}
			\begin{split}
				I_1 =  \left[ \iint_{\e \le |x-x'|\le 1} + \iint_{ |x-x'|\le \e} \right]  \log \frac{1}{|x-x'|} \sqrt{rr'}g(x,t)W(x',t)\,dx\,dx'	
				=:I_{11}+I_{12}
			\end{split}
		\end{equation*} 
		and $I_{11}$ can be bounded directly:
		$$
		I_{11}\lesssim |\log \eps| \iint \sqrt{rr'}g(x,t)W(x',t)\,dx\,dx' \lesssim |\log \eps|\int (1+r^2)g(x,t)\,dx.
		$$
		For $I_{12}$, using $\frac{W(x,t)}{r} \lesssim \frac{1}{\e^2}$ and $r'\le r+\e$, we bound 
		\begin{equation*}\begin{aligned}
				I_{12} & \lesssim \frac{1}{\e^2} \int \sqrt{r}g(x,t)\int_{|x-x'|\le \e} (r')^{\frac32} \log \frac{1}{|x-x'|} \,dx' \,dx \\
				& \lesssim \frac{1}{\e^2} \int \sqrt{r}(r+\e)^{\frac32}g(x,t)\int_{|x-x'|\le \e}  \log \frac{1}{|x-x'|} \,dx' \,dx \\
				& \lesssim |\log \eps|\int \sqrt{r}(r+\e)^{\frac32}g(x,t) \,dx \lesssim |\log\eps|\int(1+r^2)g(x,t)\,dx. 
		\end{aligned}\end{equation*} This finishes the proof. \end{proof}
	
	\begin{lemma} Under Assumptions \ref{as initial data0} and \ref{as initial data}, there exists $C_1>0$ independent of $\e$ such that
		\begin{equation*}\label{eq sec2 A3}
			A(t):= \int \langle z \rangle r^2W(x,t)\,dx\le A(0)+C_1t \qquad \mbox{for all}\quad t \ge 0. 
		\end{equation*}
	\end{lemma}
	\begin{proof}
		It suffices to prove that
		$
		\frac{dA(t)}{dt} \lesssim 1.
		$
		A direct calculation gives (omitting $t$ for simplicity)
		\begin{equation*}\begin{aligned}
				\frac{dA(t)}{dt}&=-V_{\e} \int \frac{z}{\la z \ra} r^2W \,dx+ \int \frac{z}{\la z\ra} r^2U^zW\,dx +2\int r\la z \ra U^rW\,dx =: I_1+I_2+I_3.
		\end{aligned}\end{equation*}
		
		\medskip
		\noindent\textbf{Estimates for $I_1$.} Using Assumption \ref{as initial data0}, we compute that
		\begin{equation*}\begin{aligned}
				I_1 =&-V_{\e}\int \left( \frac{z}{\la z \ra}-\frac{z_*}{\la z_* \ra} \right)r^2W\,dx-V_{\e}\frac{z_*}{\la z_* \ra}\int r^2W\,dx \\
				=&-V_{\e}\int_{|z-z_*|\le \e^{\frac12}} \left( \frac{z}{\la z \ra}-\frac{z_*}{\la z_* \ra} \right)r^2W\,dx-V_{\e}\int_{|z-z_*|\ge \e^{\frac12}} \left( \frac{z}{\la z \ra}-\frac{z_*}{\la z_* \ra} \right)r^2W\,dx - V_{\e}\frac{z_*}{\la z_* \ra} \\
				=&-V_{\e}\frac{z_*}{\la z_* \ra}+O(\e^{\frac12}V_{\e})-V_{\e}\int_{|z-z_*|\ge \e^{\frac12}} \left( \frac{z}{\la z \ra}-\frac{z_*}{\la z_* \ra} \right)r^2W\,dx.
		\end{aligned}\end{equation*}
		It then follows from \eqref{eq sec2 A0} that
		\begin{equation}\label{eq sec2 A23}
			I_1=-V_{\e}\frac{z_*}{\la z_* \ra}+O(1).
		\end{equation}
		
		\medskip
		\noindent\textbf{Estimates for $I_3$.} Using  anti-symmetry, we may rewrite
		\begin{equation*}\begin{aligned}
				I_3=& \frac{1}{\pi}\iint \la z \ra F_1(s) \frac{z'-z}{\sqrt{rr'}}W(x)W(x')\,dx\,dx'  = \frac{1}{2\pi}\iint \frac{\left(\la z \ra-\la z' \ra \right)}{\sqrt{rr'}} \sqrt{rr'}F_1(s) \frac{z'-z}{\sqrt{rr'}}W(x)W(x')\,dx\,dx'.
		\end{aligned}\end{equation*}
		Note that Lemma \ref{le se1 Biot savart estimates} gives $s^2|F_1(s)| \lesssim 1$, which implies
		\begin{equation}\begin{aligned}\label{eq sec2 A25}
				|I_3| \lesssim\iint \sqrt{rr'}W(x)W(x')\,dx \lesssim1.
		\end{aligned}\end{equation}
		
		\medskip
		\noindent\textbf{Estimates for $I_2$.} We decompose
		\begin{equation*}\begin{aligned}
				I_2=& \frac{1}{2\pi} \iint F_1(s) \frac{r-r'}{r\sqrt{rr'}}\frac{z}{\la z \ra} r^2W(x)W(x')\,dx\,dx'\\
				&+\frac{1}{2\pi} \iint F_2(s) \frac{1}{r}\sqrt{\frac{r'}{r}}\frac{z}{\la z \ra}r^2 W(x)W(x')\,dx\,dx' =: I_{21}+I_{22}.
		\end{aligned}\end{equation*}
		For $I_{21}$, we have
		$$
		I_{21}=\frac{1}{2\pi} \iint F_1(s) \frac{r-r'}{\sqrt{rr'}}\left(\frac{z}{\la z \ra} r-\frac{z'}{\la z' \ra} r' \right)W(x)W(x')\,dx\,dx'.
		$$
		Recall from Lemma \ref{le se1 Biot savart estimates} that $s^2|F_{1}(s)|\lesssim 1$, we obtain
		\begin{equation*}\begin{aligned}
				|I_{21}| \lesssim& \iint F_1(s) \frac{|r-r'|}{\sqrt{rr'}}\left(|r-r'|+r'|z-z'| \right)W(x)W(x')\,dx\,dx' \\
				\lesssim& \iint s^2F_1(s) (1+r')\sqrt{rr'}W(x)W(x')\,dx\,dx'
				\lesssim \iint (1+r^2+(r')^2)W(x)W(x')\,dx\,dx' \lesssim1.
		\end{aligned}\end{equation*}
		For $I_{22}$, we estimate
		\begin{equation*}\begin{aligned}
				I_{22}=&\frac{1}{2\pi} \iint F_2(s) \sqrt{rr'}\frac{z}{\la z \ra} W(x)W(x')\,dx\,dx' \\
				=&\frac{z_*}{\la z_*\ra}\frac{1}{2\pi} \iint F_2(s) \sqrt{rr'} W(x)W(x')\,dx\,dx' \\
				&+\frac{1}{2\pi} \iint F_2(s) \sqrt{rr'}\left(\frac{z}{\la z \ra}-\frac{z_*}{\la z_* \ra}\right) W(x)W(x')\,dx\,dx' =: I_{221}+I_{222}.
		\end{aligned}\end{equation*}
		It follows from Lemma \ref{le se1 Biot savart estimates} that $|F_2(s)-\frac{1}{2}F(s^2)|\lesssim 1$, which yields
		\begin{equation*}\begin{aligned}
				I_{221}=& \frac{z_*}{\la z_*\ra}\frac{1}{4\pi} \left(\iint F(s^2) \sqrt{rr'} W(x)W(x')\,dx\,dx'+O(1)\right) \\
				=& \frac{z_*}{2\la z_*\ra}E(0)+O(1)
				=V_{\e}\frac{z_*}{\la z_*\ra}+O(1).
		\end{aligned}\end{equation*}
		For $I_{222}$, it follows from Lemma \ref{le app2 A0} that
		$$
		|I_{222}| \lesssim |\log\eps| \int (1+r^2)g(x,t)\,dx, \qquad\mbox{where}\qquad
		g(x,t):=\left|\frac{z}{\la z \ra}-\frac{z_*}{\la z_* \ra}\right| W(x,t).
		$$
		Using Assumption \ref{as initial data0} and  \eqref{eq sec2 A0}, we obtain $|I_{222}|\lesssim 1$, which proves
		\begin{equation*} 
			I_2=V_{\e}\frac{z_*}{\la z_*\ra}+O(1).
		\end{equation*}
		Combining this with \eqref{eq sec2 A23} and \eqref{eq sec2 A25}, the proof is complete.
	\end{proof}

	\begin{proof}[{Proof of Theorem \ref{thm dynamical stable}}]
		We now prove part (b) of Theorem \ref{thm dynamical stable}, which together with Theorem \ref{thm sec2 main0} concludes the proof of Theorem \ref{thm dynamical stable}. Recalling that $z_*(t)=z^*(t)-V_{\e}t$, it suffices to prove that $|z_*(t)|\lesssim 1+A(0)+t$. First we observe that
		\begin{equation*}\begin{aligned}
				A(t) & = \int \langle z \rangle r^2W(x,t)\,dx \ge \int_{|z-z_*|\le \e^{\frac12}} \la z \ra r^2 W(x,t)\,dx \\
				& \ge \la z_* \ra \int_{|z-z_*|\le \e^{\frac12}} r^2 W(x,t)\,dx-C\e^{\frac12}  
				\ge \la z_* \ra (1-\frac{C_0}{|\log\eps|})-C\e^{\frac12}.
		\end{aligned}\end{equation*} Combining this  with the previous lemma which gives 
		$
		\int \langle z \rangle r^2W(x,t)\,dx \le A(0)+C_1t,
		$
		we get
		\begin{equation}\label{eq sec2 z*}
			|z_*(t)|^2 \le  {\left(\frac{A(0)+C_1t+C\e^{\frac12}}{1-{C_0}{|\log \eps|^{-1}} } \right)^2-1  },
		\end{equation}
		which implies 
		$
		|z_*(t)| \lesssim 1+A(0)+t. 
		$  \end{proof}

	
	\section{Filamentation and instability}\label{sec:unstable}
	
	The aim of this section is to obtain filamentation and instability. Throughout this section, we take $r_{0} = 1 = \mu$ and suppress the subscript $\eps$ from the data and solutions for simplicity. 
	
	\subsection{Proof of filamentation} In this subsection, we state and prove a slightly more precise version of Theorem \ref{thm filamentation}. 
	
	\begin{theorem}\label{thm sec3 filamentation}
		For $c_{1}, c_{2}, c_{3}, c_{4} > 0$, there exist $\eps_{0},\ell_{0}>0$ such that the following holds. Assume that $w_{0}$ satisfies  \textup{(i)}--\textup{(iv)} from Theorem \ref{thm dynamical stable} for some $0<\eps\le\eps_{0}$ with $r_{0} = 1 = \mu$ and admits a decomposition of the form \begin{equation*}\label{eq:fila-assume2}
			\begin{split}
				w_{0} = w_{m,0} + w_{d,0} \quad \mbox{with} \quad w_{m,0}, w_{d,0} \ge 0 \quad\mbox{and} \quad 0 < \left\Vert \frac{w_{d,0}}{r} \right\Vert_{L^{\infty}} \le \frac{|\log \eps|^{2}}{c_{d}}\nrm{ w_{d,0} }_{L^{1}} ,  
			\end{split}
		\end{equation*} for some $ c_{d} > 0$. For the solution $w(\cdot,t)$ corresponding to the initial datum $w_{0}$, denote by $u(\cdot,t)$ the associated velocity. Then, $w_{d}(\cdot,t)$ defined by the solution of \begin{equation*}
			\begin{split}
				\rd_{t} w_{d} + \nb \cdot ( u w_{d} ) = 0, \qquad w_{d}(t=0) = w_{d,0}
			\end{split}
		\end{equation*} satisfies, with $Z_{d}(t) := \nrm{ w_{d}(\cdot,t)}_{L^{1}}^{-1} \int_{\bbH} z w_{d}(x,t) dx$, \begin{equation}\label{eq:Z-d-slow}
			\begin{split}
				Z_{d}(t) \le Z_{d}(0) + \ell_{0}( c_{d}^{-\frac12} |\log\eps| + c_{d}^{-1} + 1) t \qquad \mbox{for all}\qquad t\ge0. 
			\end{split}
		\end{equation} 
	\end{theorem}
	
	We shall need a simple lemma, which is proved in the Appendix.
	\begin{lemma}\label{le appendix A2}
		For any $x' \in \bbH$, we have the bound 
		\begin{equation*}
			\left|\int_{\mathbb{H}}\frac{1}{|x'-x|}f(x)\,dx\right| \lesssim (1+r')\|f\|_{L^1}^{\frac12}\left\|\frac{f}{r}\right\|_{L^{\infty}}^{\frac12}+\|f\|_{L^1}.
		\end{equation*}
	\end{lemma}
	
	\begin{proof}[Proof of Theorem \ref{thm sec3 filamentation}]
		We decompose  
		\begin{equation*}\begin{aligned}
				\frac{d}{dt} \int zw_d(x,t)\,dx & = \int u^z(x,t)w_d(x,t)\,dx  = \iint K^z(x,x')w_d(x,t) \,w(x',t) \,dx \,dx' = \sum_{i=1}^3I_i,
			\end{aligned}
		\end{equation*}
		where
		$$
		I_i:= \iint_{\Omg_{i}} K^z(x,x')w_d(x,t) \,w(x',t) \,dx \,dx' = \int_{\bbH} \, \left( \int_{ \Omg_{i} \cap (\bbH \times \{ x' \}) }  K^z(x,x')w_d(x,t) \, dx \right)  w(x',t) \, dx' 
		$$
		with  \begin{equation*}
			\begin{split}
				\Omg_{1} & := \left\{ (x,x') \,: \, r \in [1/3,3] \right\}, \\  \Omg_{2} &  := \left\{ (x,x') \,: \, r \in (0,1/3) \cup (3,\infty) , \, r' \in [1/2,2]  \right\}, \\ \Omega_3 & :=\mathbb{H}\times \mathbb{H}\setminus (\Omega_1 \cup \Omega_2).
			\end{split}
		\end{equation*} 
		Using \eqref{eq biot velocity}, we obtain 
		\begin{equation}\label{eq sec3 M001}
			|K(x,x')|\lesssim  |x-x'|\left| F_1(s)r^{-\frac32}r'^{-\frac12} \right|+\left| F_2(s)r^{-\frac32}r'^{\frac12} \right|.
		\end{equation}
		
		\noindent
		\textbf{Estimate for $I_1$.} We consider $(x,x') \in \Omg_{1}$. Lemma \ref{le se1 Biot savart estimates} gives for every $s$ that 
		\begin{equation}\label{eq sec3 M002}
			|F_1(s)| \lesssim \frac{rr'}{|x-x'|^2} \quad \mbox{and} \quad |F_2(s)|\lesssim \frac{\sqrt{rr'}}{|x-x'|},
		\end{equation}
		which implies, recalling that $r \in [1/3,3]$ in $\Omega_1$,
		$$
		|K(x,x')|\lesssim \frac{r'^{\frac12}+r'}{|x-x'|}\lesssim \frac{1+r'}{|x-x'|}.
		$$
		Therefore, Lemma \ref{le appendix A2} yields
		\begin{equation}\begin{aligned}\label{eq sec3 N009}
				|I_1|  & \lesssim \|w_d\|_{L^1}^{\frac12}\left\| \frac{w_d}{r} \right\|_{L^{\infty}}^{\frac12} \, \int (1+r')^2 w(x',t)\,dx' +\int(1+r')w(x',t)\,dx'\int w_d(x,t)\,dx \\
				&\lesssim \|w_d\|_{L^1}^{\frac12}\left\| \frac{w_d}{r} \right\|_{L^{\infty}}^{\frac12}+\|w_d\|_{L^1}.
		\end{aligned}\end{equation}
		
		\noindent
		\textbf{Estimate for $I_2$.} We consider $(x,x') \in \Omg_{2}$; note that $|x-x'|\gtrsim 1$. Lemma \ref{le se1 Biot savart estimates} gives
		\begin{equation}\label{eq sec3 M004}
			|F_1(s)| \lesssim \frac{r^{\frac32}r'^{\frac32}}{|x-x'|^3} \quad \mbox{and} \quad |F_2(s)|\lesssim \frac{r^{\frac32}r'^{\frac32}}{|x-x'|^3}.
		\end{equation}
		Combining $|x-x'|\gtrsim 1$ with \eqref{eq sec3 M001}, we obtain $|K(x,x')| \lesssim r' +r'^2$. Thus,
		\begin{equation}\begin{aligned}\label{eq sec3 N008}
				|I_2| \lesssim \int w_d(x,t)\,dx \int r'w(x',t)\,dx'+\int w_d(x,t)\,dx \int (r')^2w(x',t)\,dx'
				\lesssim \|w_d\|_{L^1}.
		\end{aligned}\end{equation}
		
		\noindent
		\textbf{Estimate for $I_3$.}
		For some $R \ge 0$ to be determined, we decompose $\Omega_3 = \Omega_{3,1}\cup \Omega_{3,2}$
		where
		$$
		\Omega_{3,1} := \Omega_3 \cap \left\{ (x,x') \, : \, |x-x'|\le R \right\} \quad \mbox{and} \quad \Omega_{3,2} := \Omega_3 \cap \left\{ (x,x') \, : \, |x-x'|> R  \right\}.
		$$
		Then we set $I_3=I_{3,1}+I_{3,2}$, correspondingly. For $I_{3,1}$, using \eqref{eq sec3 M001} and \eqref{eq sec3 M002}, we get
		\begin{equation*}
			\begin{aligned}
				|I_{3,1}| & \lesssim \iint_{\Omega_{3,1}} \left( \frac{r'^{\frac12}r^{-\frac12}}{|x-x'|}+\frac{r'r^{-1}}{|x-x'|} \right)w_d(x,t) w(x',t)\,dx\,dx'\\
				&\lesssim \iint_{\Omega_{3,1}} \frac{1+r'r^{-1}}{|x-x'|} w_d(x,t) w(x',t)\,dx\,dx'\\
				&= \iint_{\Omega_{3,1}} \frac{r'r^{-1}}{|x-x'|} w_d(x,t) w(x',t)\,dx\,dx'+\iint_{\Omega_{3,1}} \frac{1}{|x-x'|} w_d(x,t) w(x',t)\,dx\,dx' =: I_{3,1,1}+I_{3,1,2}.
			\end{aligned}
		\end{equation*}
		For $I_{3,1,1}$, we have 
		\begin{equation}\label{eq sec3 N000}
			\begin{aligned}
				I_{3,1,1} &\lesssim \int_{r'<\frac12 \,\, or \,\, r'>2} \left( \int_{|x-x'|\le R} \left\| \frac{w_d}{r} \right\|_{L^{\infty}} \frac{1}{|x-x'|} \,dx \right) \, r'w(x',t)\,dx'\\
				&\lesssim R\left\| \frac{w_d}{r} \right\|_{L^{\infty}}\int_{r'<\frac12 \,\, or \,\, r'>2} r'w(x',t)\,dx' 
				\lesssim R|\log \e|^{-1}\left\| \frac{w_d}{r} \right\|_{L^{\infty}},
			\end{aligned}
		\end{equation}
		where we have used Theorem \ref{thm dynamical stable} in the last inequality. For $I_{3,1,2}$, it follows from Lemma \ref{le appendix A2} that
		\begin{equation}\begin{aligned}\label{eq sec3 N001}
				|I_{3,1,2}| \lesssim& \int (1+r') w(x',t)\,dx'\|w_d\|_{L^1}^{\frac12}\left\| \frac{w_d}{r} \right\|_{L^{\infty}}^{\frac12}+\int w(x',t)\,dx'\int w_d(x,t)\,dx \\
				\lesssim&\, \|w_d\|_{L^1}^{\frac12}\left\| \frac{w_d}{r} \right\|_{L^{\infty}}^{\frac12}+\|w_d\|_{L^1}.
		\end{aligned}\end{equation}
		For $I_{3,2}$, using \eqref{eq sec3 M001}, \eqref{eq sec3 M004} and Theorem \ref{thm dynamical stable}, we get
		\begin{equation}\label{eq sec3 N002}
			\begin{aligned}
				|I_{3,2}| \lesssim& \iint_{\Omega_{3,2}} \left( \frac{r'}{|x-x'|^2}+\frac{r'^2}{|x-x'|^3} \right)w_d(x,t) w(x',t)\,dx\,dx'\\
				\lesssim&\left(R^{-2}+R^{-3}\right)|\log \e|^{-1}\|w_d\|_{L^1} 
				\lesssim \left(1+R^{-3}\right)|\log \e|^{-1}\|w_d\|_{L^1}.
			\end{aligned}
		\end{equation}
		Set $R=\|w_d\|_{L^1}^{\frac14}\| \frac{w_d}{r}\|_{L^{\infty}}^{-\frac14}$, together with \eqref{eq sec3 N000}, \eqref{eq sec3 N001} and \eqref{eq sec3 N002}, we obtain
		\begin{equation}\label{eq sec3 N007}
			|I_3| \lesssim 
			\|w_d\|_{L^1}^{\frac14}\| \frac{w_d}{r}\|_{L^{\infty}}^{\frac34}|\log \e|^{-1}+\|w_d\|_{L^1}^{\frac12}\left\| \frac{w_d}{r} \right\|_{L^{\infty}}^{\frac12}+\|w_d\|_{L^1}.
		\end{equation}
		Using \eqref{eq sec3 N009}, \eqref{eq sec3 N008} and \eqref{eq sec3 N007}, we arrive at
		\begin{equation*}
			\left|\frac{d}{dt} \int zw_d(x,t)\,dx \right| \lesssim |I_1|+|I_2|+|I_3| \lesssim 
			\|w_d\|_{L^1}^{\frac14}\| \frac{w_d}{r}\|_{L^{\infty}}^{\frac34}|\log \e|^{-1}+\|w_d\|_{L^1}^{\frac12}\left\| \frac{w_d}{r} \right\|_{L^{\infty}}^{\frac12}+\|w_d\|_{L^1},
		\end{equation*}
		which implies
		\begin{equation*}\label{eq sec3 main}
			| \dot{Z}_{d} | \lesssim \|w_d\|_{L^1}^{-\frac34}\| \frac{w_d}{r}\|_{L^{\infty}}^{\frac34}|\log \e|^{-1}+\|w_d\|_{L^1}^{-\frac12}\left\| \frac{w_d}{r} \right\|_{L^{\infty}}^{\frac12}+1.
		\end{equation*}
		Now based on our assumption on $w_{d,0}$, we conclude that 
		\begin{equation*}
			|\dot{Z}_{d}(t)| \lesssim c_{d}^{-\frac34}|\log \epsilon|^{\frac12}+c_{d}^{-\frac12}|\log \epsilon|+1 \lesssim c_{d}^{-\frac12}|\log \epsilon| + c_{d}^{-1} + 1
		\end{equation*} and we denote by $\ell_{0}$ the implicit constant, which does not depend on $c_{d}$ and $\eps$.
	\end{proof}
	
	\begin{proof}[Proof of Theorem \ref{thm filamentation}]
		We assume $r_{0} = 1 = \mu$ in the proof. From the definition of $Z_{d}(t)$ and \eqref{eq:Z-d-slow}, for any $t$ we deduce that there is a point $x' = (r', z') \in \supp(w_{d}(\cdot,t))$ such that  $z' \le Z_{d}(0) + \ell_{0}c_{d}^{-1/2} |\log\eps|t$. On the other hand, from Theorem \ref{thm dynamical stable}, we know existence of a point $x^{**} = (r^{**}, z^{**}) \in \supp(w(\cdot,t))$ such that $z^{**} \ge \frac{1}{4\pi}|\log \e|t-C_0(1+A_{0,\e}+t)$. In particular, we have that \begin{equation*}
			\begin{split}
				\mathrm{diam}_{z}( \supp(w(\cdot,t)) ) \ge z^{**} - z' \ge \left( \frac{1}{4\pi } - \frac{1 + \ell_{0}(1+c_{d}^{-1})}{|\log \e|} - \ell_{0}c_{d}^{-1/2} \right)|\log\eps|t - Z_{d}(0) - C_{0}(1+A_{0,\e}),
			\end{split}
		\end{equation*}
		concluding the proof provided that $c_d$ is chosen large enough.  
	\end{proof} 
	\begin{proof}[Proof of Theorem \ref{thm filamentation without perturbation}]
		Again, we take $r_{0} = 1 = \mu$. By our assumption in Theorem \ref{thm filamentation without perturbation}, there exist $V \subset U \cap \{ r > c_{5} \}$ and $\eta>0$ such that
		\[
		|V| \ge \frac{\tilde{C}_d}{2 |\log \eps|^2} \quad \text{and} \quad w_{0,\eps}(x) \ge \eta \quad \text{for all } x \in V.
		\]
		We then decompose
		\[
		w_{0,\eps} = \bigl(w_{0,\eps} - \eta \mathbf{1}_V\bigr) + \eta \mathbf{1}_V := w_{m,0,\eps} + w_{d,0,\eps}.
		\]
		It is straightforward to check that
		\[
		0 < \Bigl\| \frac{w_{d,0,\eps}}{r} \Bigr\|_{L^\infty} \le \frac{\eta}{c_5} \le \frac{2 |\log \eps|^2}{c_5 \tilde{C}_d} \, \| w_{d,0,\eps} \|_{L^1}.
		\]
		Choosing $\tilde{C}_d := \frac{2 c_d}{c_5}$ where $c_{d}$ is from the proof of Theorem \ref{thm filamentation} above, the conclusion follows directly from Theorem~\ref{thm dynamical stable} and Theorem~\ref{thm sec3 filamentation}.
	\end{proof}
	
	\subsection{Proof of norm growth} We are now ready to prove Corollary \ref{co filamentation without perturbation} and  Theorem \ref{thm dynamical unstable}. 
	\begin{proof}[Proof of Corollary \ref{co filamentation without perturbation}]
		We take another smooth and simple closed curve $\Gamma_{2}$ such that 
		\begin{itemize}
			\item  $\Gamma_{2}\cap \supp( w_{0,\e} )=\emptyset$ and is supported away from $\{ r = 0 \}$. 
			\item $\overline{D_{1}} \subset D_{2}$, where $D_{j}$ is the open bounded region bounded by the curve $\Gamma_{j}$ for $j = 1, 2$.  
		\end{itemize}
		For simplicity, we write $\xi(r,z,t) := w(r,z,t)/r$, so that it is transported by $u$. It is easy to see that $\nrm{ \rd_{r} \xi }_{L^{\infty}} \le 2\nrm{\rd_{rr} w }_{L^{\infty}}$, and therefore it suffices to show the growth of $\nrm{ \rd_{r} \xi }_{L^{\infty}}$. For any $t > 0$, we define $\Gmm_{j, t} := X( \Gmm_{j}, t )$ for $j = 1, 2$, namely the image of $\Gamma_{j}$ at time $t$ by the flow $X$ generated by $u$. 
		
		For each $t$, we may assume that $\Gmm_{1, t} \subset \left\{ r \ge (t|\log \eps|)^{-1/2} \right\}$, since otherwise we are done by the mean value theorem using $\xi(\Gmm_{1,t},t) \equiv \eta/c_5$ and $\lim_{r\to 0^+} {\xi(r,z,t)} = 0$. Now taking $D_{j,t}$ to be the open set bounded by $\Gamma_{j, t}$,  we have $\overline{D_{1,t}} \subset D_{2,t}$ for all $t>0$. From the conclusion of Theorems \ref{thm dynamical stable} and \ref{thm sec3 filamentation}, we know that there exist two points $x^* = (r^*,z^*)$ and $x^{**} = (r^{**},z^{**}) \in D_{1,t} $ such that $z^* \ge |\log \eps|t/(4\pi)$ while $z^{**} \le |\log \eps|t/(8\pi)$, for all sufficiently large $t$. Then for any $z \in [z^{**},z^*]$, we define \begin{equation*}\label{eq2:R-t-z}
			\begin{split}
				R_{t}(z) := (\bbR_+ \times \left\{ z \right\}) \cap D_{1, t} \ne \emptyset 
			\end{split}
		\end{equation*}
		and define $r_{t}(z)$ by the leftmost point in the set $R_{t}(z)$. Since $D_{1,t} \subset D_{2,t}$, we may consider the rightmost point $r'_{t}(z)$ in $\Gamma_{2,t} \cap ( (0, r_{t}(z)  ) \times \{ z \}  )$ and denote by $J_{t}(z)$ as the interval connecting $(r'_{t}(z)+r_{t}(z))/2$ to $r_{t}(z)$, which satisfies 
		$J_t(z) \subset \left\{ r \ge t|\log \e|^{-1/2}/2 \right\}.$ 
		
		We now define 
		$$
		U_t:=\left\{ (r,z) \Big|\, r \in J_t(z) \quad \text{and}\quad z\in[z^{**},z^*]  \right\}.
		$$ Using $U_t \subset D_{2,t}$ and the conservation of the measure $d\nu := rdrdz$, \begin{equation*}
			\begin{split}
				\frac{(t|\log\eps|)^{-1/2}|U_{t}|}{2} \le \int_{ U_{t} } r drdz \le \int_{ D_{2,t} } r drdz= \nu(D_{2}) ,   
			\end{split}
		\end{equation*} 
		which implies $|U_{t}| \le 2\nu(D_2)(t|\log\eps|)^{1/2}$.

		Since $\xi(\Gmm_{1,t},t) = \eta/c_5$ and $\xi(\Gmm_{2,t},t) = 0$, $\nrm{\rd_{r}\xi(\cdot,t)}_{L^{\infty}} \ge \eta\left(2c_5 |J_{t}(z)|\right)^{-1}$ for each $z \in I_{t}$. Then, \begin{equation*}
			\begin{split}
				2\nu(D_2)(t|\log\eps|)^{1/2}\ge |U_{t}| \ge \int_{z^{**}}^{z^*} |J_{t}(z)| dz \ge  \frac{\eta |\log\e| t}{16 \pi c_5 } \nrm{\rd_{r}\xi(\cdot,t)}_{L^{\infty}}^{-1},
			\end{split}
		\end{equation*}which finishes the proof. 
	\end{proof}
	\begin{proof}[Proof of Theorem \ref{thm dynamical unstable}] 
		It follows from Theorem \ref{thm sec2 main0} that, there exists $x^*=(r^*,z^*)$ with $|r^*-1| \lesssim |\log \eps|^{-1}$ such that \begin{equation*}
			\begin{split}
				\int_{D_1} w_{0,\e}(x)\,dx \ge \frac{1}{2}, \quad \mbox{where} \quad D_1=B_{\frac{C_{2,d}}{|\log \e|}}(x^*)
			\end{split}
		\end{equation*} with $C_{2,d}>0$ being the constant from Theorem \ref{thm filamentation without perturbation} associated with $2c_1, 2c_2, 2c_3, 2c_4$. Next, we define $w_{0} := w_{\star, 0}+\delta w_{\star, 0 ,p}$ and $w_{\star, 0 ,p} := \phi_1+\phi_2$,  where
		\begin{itemize}
			\item $\phi_1$ is a smooth function that is supported in $B_{\frac{2C_{2,d}}{|\log \e|}}(x^*)$ and satisfying $\phi_1 \equiv 1$ in $ D_1$.
			\item Take $\Gamma_{1}$ to be a smooth and simple closed curve containing $\supp( w_{\star,0} + \phi_1 )$ and is away from $\{ r = 0 \}$. Then, we take $\phi_2$ to be a smooth function such that $\phi_2\equiv 1$ in $\Gamma_1$ and supported away from  $\supp( w_{\star,0}+\phi_1 )$.
		\end{itemize}
		The proof is completed by Corollary \ref{co filamentation without perturbation}, provided that $\delta>0$ is chosen small enough.
	\end{proof}

	
	\section{Open Problems}\label{sec:open}
	
	In this final section, we present some related open problems for vortex rings.
	
	\medskip
	\noindent\textbf{Problem 1. Sharp asymptotics of the translation velocity.} 
	Theorem \ref{thm dynamical stable} establishes that for a general class of initial data, the axial translation follows the law $z_{\epsilon}^*(t) = \left( V_{\epsilon}+O(1)\right) t $ as $t \to \infty$. A natural and deep question is whether the $O(1)$ term can be further refined.
	
	For specific steady vortex ring profiles (e.g. those of the form \eqref{eq fraenkel profile} constructed by Fraenkel \cite{Fraenkel1970}), the translation speed is known to admit an expansion of the form $V_{\epsilon} + C_{\zt} + o(1)$, where $C_{\zt}$ is a constant explicitly depending on the internal profile $\zt$. In our setting, where the initial configuration is not necessarily near a steady state, it remains to be determined whether the $O(1)$ correction is uniquely determined by the macroscopic conserved quantities $(M_0, M_2, E)$, or if it retains a persistent dependence on the microscopic distribution of the initial vorticity. 
	
	\medskip
	\noindent\textbf{Problem 2. Optimal geometric threshold for filamentation.} 
	We identified a critical length scale of $O(|\log \epsilon|^{-1})$ for the thickness of the vorticity: beyond this scale, the mismatch between the core's translation and the peripheral induction triggers linear filamentation. This raises the question of the sharpness of this threshold, namely whether filamentation always occur for strictly ``thinner'' vorticities or not. 
	
	\medskip
	\noindent\textbf{Problem 3. Long time dynamics of the ring.} 
	While we prove that a portion of the vorticity is ``shed'' from the core and experiences linear-in-time stretching, its ultimate asymptotic behavior as $t \to \infty$ is unknown. Given the measure-preserving nature of the flow, the filaments must become increasingly thin. One might ask whether these filaments eventually ``mix'' into a diffuse background in the sense of weak-$*$ convergence in $L^\infty$, or if its self-induction leads to a secondary re-organization into smaller, coherent vortex structures. Understanding this cascade or re-aggregation within the framework of the three dimensional axisymmetric Euler remains a challenging open problem.
	
	\appendix
	\section{}
	
	We provide the proofs of a few technical lemmas.
	
	
	\subsection{Proofs of Lemmas \ref{le app B1} and \ref{le feng-sverak}}

	\begin{proof}[Proof of Lemma \ref{le app B1}]
		We estimate, with $s=\frac{|x-x'|}{\sqrt{rr'}}$,
		\begin{equation*}
			\begin{aligned}
				E(w) &  
				\le \frac{1}{2\pi}\iint_{s \le 1} \log \frac{\sqrt{rr'}}{|x-x'|}\sqrt{rr'}w(x)w(x')\,dx\,dx' +C\iint_{s>1} \sqrt{rr'}w(x)w(x')\,dx\,dx' \\
				& \le \frac{1}{2\pi}\iint_{|x-x'| \le 1} \log \frac{1}{|x-x'|}\sqrt{rr'}w(x)w(x')\,dx\,dx' +\frac{1}{2\pi}\iint_{\sqrt{rr'} \ge 1} \sqrt{rr'}\log \sqrt{rr'}\,w(x)w(x')\,dx\,dx'\\
				& \quad + C \iint \sqrt{rr'}w(x)w(x')\,dx\,dx'.
			\end{aligned}
		\end{equation*}
		Using 
		$$
		\sqrt{rr'}\log\sqrt{rr'}\lesssim 1+r^2+(r')^2 \qquad \mbox{		and} \qquad  \sqrt{rr'} \le \frac{2+r^2+(r')^2}{4},
		$$
		it follows from Assumption \ref{as initial data0}, conservation of the total vorticity and second momentum that
		\begin{equation*} 
			E(w)\le  E_1(w)+C\iint(1+r^2+(r')^2)w(x)w(x')\,dx\,dx'
			\le E_1(w)+C.   \qedhere  
		\end{equation*}
	\end{proof}

	\begin{proof}[Proof of Lemma \ref{le feng-sverak}]
		We need to prove  
		$$
		\|Gf\|_{L^{\infty}} \lesssim \|r^2f\|_{L^1}^{\frac14}\|f\|_{L^1}^{\frac14}\|\frac{f}{r} \|_{L^{\infty}}^{\frac12} 
		, \quad\mbox{for}\quad Gf(x) := \int_{\bbH} \log \frac{1}{|x-x'|}\mathbf{1}_{\{|x-x'|\le 1\}}f(x')\,dx'.
		$$ 
		We observe that $$ \log \frac{1}{|x-x'|}\mathbf{1}_{\{|x-x'|\le 1\}} \lesssim {|x-x'|^{-\frac{1}{100}}}.$$ Using generalized H\"older inequality, we obtain
		\begin{equation*}
			\begin{aligned}
				|Gf(x)|\lesssim& \int_{|x-x'|\le 1} |x-x'|^{-\frac{1}{100}} |r^2f|^{\frac14}\left| \frac{f}{r} \right|^{\frac12} |f|^{\frac14}\,dx' \\
				\lesssim& \|r^2f\|_{L^1}^{\frac14}\|f\|_{L^1}^{\frac14}\|\frac{f}{r} \|_{L^{\infty}}^{\frac12} \left(\int_{|x-x'|\le 1} |x-x'|^{-\frac{1}{50}}\,dx'\right)^{\frac12} \lesssim  \|r^2f\|_{L^1}^{\frac14}\|f\|_{L^1}^{\frac14}\|\frac{f}{r} \|_{L^{\infty}}^{\frac12}. 
			\end{aligned}
		\end{equation*} This finishes the proof. 
	\end{proof}
	

	\subsection{Proof of Lemma \ref{le appendix A2}}
	
	For convenience, we define $$Tf(x):=\int_{\mathbb{H}}\frac{1}{|x-x'|}f(x')\,dx'; \quad \mbox{the goal is to prove} \quad |Tf(x)| \lesssim (1+r)\|f\|_{L^1}^{\frac12}\left\|\frac{f}{r}\right\|_{L^{\infty}}^{\frac12}+\|f\|_{L^1}. $$
	
	\begin{proof}[Proof of Lemma \ref{le appendix A2}]
		We decompose $f(x')=f_1(x')+f_2(x,x')+f_3(x,x')$, where
		\begin{equation*}
			\begin{aligned}
				f_{1}(x):=& f(x')\mathbf{1}_{\{r'\le 1\}}, \quad 
				f_2(x,x'):=\frac{f(x')}{r'}(r'-r)\mathbf{1}_{\{r'\ge 1\}},\quad 
				f_3(x,x'):=\frac{f(x')}{r'}r\mathbf{1}_{\{r'\ge 1\}}.
			\end{aligned}
		\end{equation*}
		Then we define $Tf=T_1f+T_2f+T_3f$, correspondingly. Recalling the well-known estimate 
		$$
		|Tg(x)| \lesssim \|g\|_{L^1}^{\frac12}\|g\|_{L^{\infty}}^{\frac12},
		$$
		it follows that 
		\begin{equation*}
			\begin{aligned}
				|T_1f(x)| \lesssim \|f\|_{L^1}^{\frac12}\left\|\frac{f}{r}\right\|_{L^{\infty}}^{\frac12}\qquad \mbox{and} \qquad 
				|T_3f(x)| \lesssim  r\|f\|_{L^1}^{\frac12}\left\|\frac{f}{r}\right\|_{L^{\infty}}^{\frac12},
			\end{aligned}
		\end{equation*}
		where we have used the fact that $|f(x')| \le |\frac{f(x')}{r'}|$ when $r' \le 1$ in the first inequality and
		$|\frac{f(x')}{r'}| \le |f(x')|$ when $r' \ge 1$ in the second inequality. For $T_2f$, it can be check directly that
		$$
		|T_2f(x)| \lesssim \|f\|_{L^1},
		$$
		which completes the proof.
	\end{proof}

	\bibliographystyle{alpha}
	\bibliography{vortexring.bib}

\end{document}